\documentclass[11pt]{article}
\usepackage[a4paper,margin=26mm]{geometry}
\usepackage[T1]{fontenc}
\usepackage{lmodern,microtype,amsmath,amssymb,amsthm,mathtools,booktabs,array}
\usepackage[dvipsnames]{xcolor}
\usepackage{tikz}
\usepackage{longtable}
\usetikzlibrary{arrows.meta,positioning,calc}
\usepackage[colorlinks=true,linkcolor=MidnightBlue,citecolor=MidnightBlue,
 urlcolor=MidnightBlue]{hyperref}
\usepackage{fancyhdr}
\numberwithin{equation}{section}
\newtheorem{theorem}{Theorem}[section]
\newtheorem{proposition}[theorem]{Proposition}
\newtheorem{lemma}[theorem]{Lemma}
\newtheorem{corollary}[theorem]{Corollary}
\theoremstyle{definition}
\newtheorem{definition}[theorem]{Definition}
\newtheorem{remark}[theorem]{Remark}
\newtheorem{example}[theorem]{Example}
\newcommand{\Q}{\mathbb Q}\newcommand{\Z}{\mathbb Z}
\newcommand{\C}{\mathbb C}\newcommand{\F}{\mathbb F}
\newcommand{\OO}{\mathcal O}\newcommand{\EE}{\mathcal E}
\newcommand{\pp}{\mathfrak p}\newcommand{\qq}{\mathfrak q}
\newcommand{\Cl}{\operatorname{Cl}}\newcommand{\Spec}{\operatorname{Spec}}
\newcommand{\nrd}{\operatorname{nrd}}\newcommand{\trd}{\operatorname{trd}}
\newcommand{\Nm}{\operatorname{N}}\newcommand{\GL}{\operatorname{GL}}
\newcommand{\PGL}{\operatorname{PGL}}
\newcommand{\one}{\mathbf 1}
\newcommand{\wh}{\widehat}
\newcommand{\End}{\operatorname{End}}

\title{The Bellaterra graphs are Ramanujan}
\author{Olga Kharlampovich, Alina Vdovina}
\date{17 October 2026}

\hypersetup{pdftitle={The Bellaterra graphs are Ramanujan},
 pdfsubject={Bellaterra automaton, cubic Ramanujan graphs, and quaternionic arithmetic},
 pdfkeywords={Bellaterra, Ramanujan, Schreier graph, quaternion algebra, Hecke operator}}
\begin{document}
\maketitle
\begin{abstract}
The Bellaterra automaton gives three involutions of every set of binary
words of a fixed length. We identify their Schreier graph at level $n$
with a cubic arithmetic graph over $\Q(\sqrt{17})$, and prove that every
nonconstant adjacency eigenvalue satisfies the Ramanujan bound.
In particular, this proves Malyshev--Pak's Conjecture
 that the Bellaterra graphs form a family of two-sided expanders.

\end{abstract}

\section{Introduction}
\label{sec:introduction}

Three permutations $a,b,c$ of the binary words of length $n$ define a
finite graph: from $w$ take one step to $a(w)$, $b(w)$, or $c(w)$.
For the Bellaterra automaton these permutations are involutions, so
the adjacency matrix is symmetric and has row sum three. The question
is whether the resulting graphs remain uniformly well connected as
$n$ grows. We prove the precise spectral estimate
\[
       |\lambda|\leq2\sqrt2
       \quad\text{for every eigenvalue on the nonconstant subspace.}
\]
Section~\ref{sec:theorem} defines the graphs and explains why this is
the cubic Ramanujan bound. The proof uses exact identities that hold
at every level, followed by a stated theorem from arithmetic.

\subsection{The expansion problem}

The problem of constructing expander or Ramanujan families from finite
automata was posed by Grigorchuk
\cite[Section~12, Problem~12.3(a,b)]{Gri05}.
Ramanujan graphs of valence $q+1$
with $q$ an odd prime power, generated by automata, were contructed by Bondarenko, Grigorchuk and Vdovina
in \cite[Theorem~2]{BGV}.
For the Bellaterra level graphs, Grigorchuk explicitly asked whether they form
an expander family
\cite[Example~10.2 and Problem~10.1, p.~155]{Gri11}.
Malyshev and Pak subsequently formulated the two-sided expansion
conjecture \cite[Conjecture~1.3]{MP}.
We prove the stronger Ramanujan estimate for the cubic adjacency
operator defined below. \subsection{A guide to the proof}

The construction compares three descriptions of the same finite set:
\begin{equation}\label{eq:roadmap}
 \{0,1\}^n
 \xrightarrow[\text{explicit coding}]{\ \kappa_n\ }
 \Z/2^n\Z
 \xleftrightarrow[\text{all ideal classes}]{\ \cong\ }
 \Cl(\EE_n).
\end{equation}
The first set carries the automaton action. The second carries
fractional-linear transformations arising from three quaternions.
The third is a finite arithmetic class set on which a Hecke operator
is defined. Here $\EE_n$ is an order, a quaternionic analogue of a
ring of integers, with a congruence condition at the level prime.
Definitions appear before these objects are used in the proof.

The first bijection \ref{eq:roadmap} is proved by six multiplication identities. The second follows from a base
class-number-one calculation. We then prove that the sum of the
three arithmetic permutations is the Hecke operator on all ideal classes.
Finally, Jacquet--Langlands and the Hilbert modular Ramanujan theorem
bound its nonconstant eigenvalues.

\section{The Bellaterra graphs and the theorem}
\label{sec:theorem}

\subsection{Rooted trees, sections, and finite automata}

The \emph{rooted binary tree} has vertex set $\{0,1\}^*$, the finite
words in two letters. Its root is the empty word. The two children of
$w$ are $w0,w1$, and its level $n$ is $V_n=\{0,1\}^n$. A rooted-tree
automorphism is a permutation of the vertices that fixes the root
and preserves adjacency. Equivalently, it preserves word length and
the operation of deleting the last letter.

Every such automorphism $g$ has a permutation $\sigma$ of the first
letter and two automorphisms $g_0,g_1$ of the remaining word:
\[
                  g(ew)=\sigma(e)g_e(w),\qquad e\in\{0,1\}.
\]
The maps $g_e$ are called its \emph{sections}. The notation
$g=(g_0,g_1)\sigma$ is called \emph{wreath recursion}; an identity
$\sigma$ is omitted. Iterating gives the section $g|_w$ at any word.
A finite invertible automaton specifies finitely many states by
their first-letter permutations and their sections, which are again
among the states. It reads one letter, writes one letter, and moves
to the indicated state for the remaining input.

For a word $w\in\{0,1\}^*$, define the Bellaterra states by
\begin{equation}\label{eq:recursion}
\begin{array}{lll}
 a(0w)=0b(w),& a(1w)=1c(w),\\
 b(0w)=0c(w),& b(1w)=1b(w),\\
 c(0w)=1a(w),& c(1w)=0a(w).
\end{array}
\end{equation}
All three fix the empty word. In wreath-recursion notation this is
\[
             a=(b,c),\qquad b=(c,b),\qquad c=(a,a)\sigma,
             \qquad \sigma=(01).
\]
This is the Bellaterra automaton, with the binary coordinates
complemented relative to \cite[Section~5]{VV}. More precisely, if
$\delta_n$ complements every letter of a word of length $n$, then
conjugating the convention $a=(c,b)$, $b=(b,c)$,
$c=(a,a)\sigma$ by $\delta_n$ gives \eqref{eq:recursion}.
Thus the convention does not affect the graphs or their spectra.

For example,
\[
 a(000)=0b(00)=00c(0)=001,\qquad a(010)=0b(10)=01b(0)=010.
\]
The recursion defines a bijection at every level by induction on word
length. The same induction shows that the states are involutions:
their squares have identity first-letter action and sections among
the preceding-level squares. We write
$\mathcal B=\langle a,b,c\rangle$ for the group they generate.

\subsection{Schreier graphs and their adjacency matrices}

If a group generated by involutions acts on a set, its \emph{Schreier
graph} has that set as vertices and one transition $v\mapsto s(v)$
for each chosen generator $s$. For a transitive action it can also
be described on the cosets $\mathcal B/H$, where $H$ is the
stabilizer of a vertex. The generating set is part of the graph data.
Two different generators may connect the same pair of vertices.

\begin{definition}
Let $V_n=\{0,1\}^n$. Write $a_n,b_n,c_n$ for the permutation matrices
of the three states on $\C^{V_n}$, with $a_ne_w=e_{a(w)}$, and put
\begin{equation}\label{eq:adjacency}
                         M_n=a_n+b_n+c_n.
\end{equation}
The Bellaterra level graph $G_n$ is the graph with this adjacency
matrix. Edges are counted with their state labels and multiplicities.
A fixed state contributes one to a diagonal entry, so a fixed-point
petal is a half-loop in the degree convention used here.
\end{definition}

Thus $M_n$ is a real symmetric matrix with nonnegative integral
entries and row sum three. A drawn half-loop contributes one diagonal
entry, so its degree contribution is one. This convention is fixed by
\eqref{eq:adjacency}; ordinary loops counted twice would define a
different operator. On functions, the same operator is
$(M_nf)(w)=f(a(w))+f(b(w))+f(c(w))$.

\begin{example}[The first level]
Both $a$ and $b$ fix $0$ and $1$, while $c$ interchanges them. Hence
\[
       M_1=\begin{pmatrix}2&1\\1&2\end{pmatrix},\qquad
       \Spec(M_1)=\{3,1\}.
\]
At each vertex the first two states supply the two fixed-state
contributions, and the third supplies the edge to the other vertex.
\end{example}

We use the inner product
$\langle f,g\rangle=\sum_{w\in V_n}\overline{f(w)}g(w)$.
The constant function $\one$ always has eigenvalue $3$. Its orthogonal
complement is
\[
              \one^\perp=\{f:V_n\to\C:\ \sum_{w\in V_n}f(w)=0\}.
\]
For any graph defined by three involutions $s_1,s_2,s_3$,
\begin{equation}\label{eq:energy}
 \langle f,(3I-M)f\rangle
       =\frac12\sum_v\sum_{i=1}^3|f(v)-f(s_i(v))|^2.
\end{equation}
It follows that the eigenspace of $3$ consists of the functions
constant on each connected component. Its dimension is therefore the
number of components. This observation will prove connectedness at
the end of the argument.

\subsection{The spectral question}

The adjacency operator of the infinite cubic tree has spectral radius
$2\sqrt2$. A connected cubic graph is called \emph{Ramanujan} when
every adjacency eigenvalue other than the trivial values $3$ and,
if present, $-3$, lies in $[-2\sqrt2,2\sqrt2]$. The same spectral
definition applies to our involution graphs, with the matrix
convention above. We prove the stronger statement that every
eigenvalue on $\one^\perp$ belongs to that interval; in particular,
$-3$ does not occur.

The \emph{normalized adjacency} or random-walk operator is $P_n=M_n/3$.
It chooses one of the three states with probability $1/3$.
A family has a uniform \emph{absolute spectral gap}, or is a family
of \emph{two-sided expanders}, if there is an $\varepsilon>0$
independent of $n$ such that
$\|P_n|_{\one^\perp}\|\leq1-\varepsilon$.
The Ramanujan estimate gives the explicit gap
$\varepsilon=1-2\sqrt2/3>0$.

\begin{theorem}[Bellaterra Ramanujan theorem]\label{thm:main}
For every $n\geq1$, the graph $G_n$ is connected, the eigenvalue $3$
of $M_n$ is simple, and
\begin{equation}\label{eq:bound}
 \Spec\bigl(M_n|_{\one^\perp}\bigr)
                       \subseteq[-2\sqrt2,2\sqrt2].
\end{equation}
In particular, the Bellaterra level graphs are cubic Ramanujan graphs
with the adjacency convention \eqref{eq:adjacency}.

\end{theorem}

The theorem will be proved in Section 8.

Deletion of the last letter commutes with the three states. Thus these
graphs form a compatible tower. The pullback map
$J_{\rm B,n}e_w=e_{w0}+e_{w1}$ satisfies
$M_{n+1}J_{\rm B,n}=J_{\rm B,n}M_n$.
The arithmetic change of coordinates will preserve this tower, as well
as the individual state labels. No torsion-free hypothesis will be
imposed on the arithmetic group.

\begin{figure}[htbp]
\centering
\begin{tikzpicture}[state/.style={circle,draw=MidnightBlue,
 fill=MidnightBlue!5,minimum size=9mm},
 every edge/.style={draw,-{Latex[length=2mm]},thick}]
 \node[state] (a) at (0,1.6) {$a$};
 \node[state] (b) at (-3,0) {$b$};
 \node[state] (c) at (3,0) {$c$};
 \path (a) edge node[above left] {$0\mid0$} (b);
 \path (a) edge node[below left,pos=.45] {$1\mid1$} (c);
 \path (b) edge node[below] {$0\mid0$} (c);
 \path (b) edge[loop left] node[left] {$1\mid1$} (b);
 \path (c) edge[bend right=35]
 node[above right,pos=.55] {$0\mid1,\ 1\mid0$} (a);
\end{tikzpicture}
\caption{The Bellaterra automaton in the convention of
\eqref{eq:recursion}. An arrow is labelled input$\mid$output, and
its target is the section acting on the remaining word.}
\end{figure}
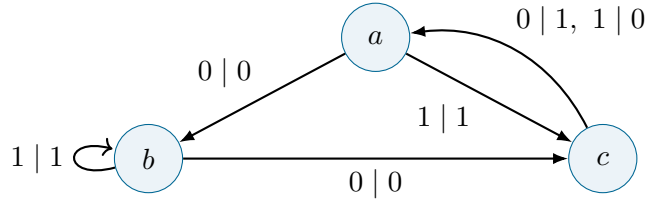

\section{Dyadic arithmetic and the two primes above two}
\label{sec:dyadic}

\subsection{Compatible residues and the dyadic integers}

The \emph{dyadic integers} can be defined by
\[
                    \Z_2=\varprojlim_m\Z/2^m\Z.
\]
An element is a sequence of residues modulo $2,4,8,\ldots$ that agree
under reduction. Addition and multiplication are performed in every
residue ring. For $x\ne0$, the valuation $v_2(x)$ is the largest
integer $m$ for which $x\equiv0\pmod{2^m}$; set $v_2(0)=+\infty$.
The distance is $|x-y|_2=2^{-v_2(x-y)}$. Thus closeness means agreement
modulo a high power of two. The fraction field is $\Q_2$. An element
of $\Z_2$ is a unit, meaning that it has an inverse in $\Z_2$, exactly
when it is odd.

Every ball $a+2^m\Z_2$ splits into two balls of half the dyadic radius:
\[
 a+2^m\Z_2=(a+2^{m+1}\Z_2)\ \sqcup\
                         (a+2^m+2^{m+1}\Z_2).
\]
These refinements form a rooted binary tree. Its level $m$ is
$\Z/2^m\Z$, with parent map given by reduction. A compatible bijection
between this tree and the word tree will be our map $\kappa_n$.
The first input letter specifies the first residue choice, hence the
least significant dyadic digit. The full coding is nonlinear: a word
must not be replaced by its usual binary integer.

We use only the following elementary form of \emph{Hensel lifting}.
If $f$ is an integer polynomial and $x_m$ is a root modulo $2^m$
with $f'(x_m)$ odd, precisely one of $x_m,x_m+2^m$ is a root modulo
$2^{m+1}$. Indeed,
\[
 f(x_m+e2^m)\equiv f(x_m)+e2^m f'(x_m)\pmod{2^{m+1}}
                  \qquad(e=0,1).
\]
A simple root modulo two therefore determines a unique dyadic root.
All the local coordinates used below can be computed in this way to
any desired finite precision.

\subsection{The real quadratic field}

A number field is a finite extension of $\Q$. Its ring of integers
consists of its elements satisfying monic integer polynomials.
For the real quadratic field below this ring has the indicated basis.
The two real embeddings send $r$ to $\sqrt{17}$ and $-\sqrt{17}$.
An element is \emph{totally positive} if its images under both
embeddings are positive. The field norm $\Nm_{F/\Q}(x)$ is the
product of its two conjugates.

Throughout, set
\begin{equation}\label{eq:parameters}
\begin{gathered}
 r=\sqrt{17},\qquad \omega=\frac{1+r}{2},\qquad
 \alpha=\frac{r-3}{2}=\omega-2,\qquad \OO_F=\Z[\omega],\\
 \eta=\frac{5+r}{2}=4+\alpha,\qquad
 \rho=\frac{5-r}{2}=1-\alpha,\\
 \alpha^2+3\alpha-2=0,\qquad \eta\rho=2,\qquad
 \qq=(\eta),\qquad\pp=(\rho).
\end{gathered}
\end{equation}
Both $\eta$ and $\rho$ are totally positive of norm $2$.
The factorization $(2)=\pp\qq$ has two distinct primes of norm $2$.
The prime $\pp$ will define the cubic Hecke operator, while $\qq$
will define the levels.

For clarity, a nonzero prime ideal $\mathfrak l$ of $\OO_F$ has norm
$\Nm\mathfrak l=|\OO_F/\mathfrak l|$, the number of elements in its
residue field. The completion $F_{\mathfrak l}$ records congruences
modulo successive powers of $\mathfrak l$, just as $\Q_2$ records
dyadic congruences. Its valuation ring, denoted $\OO_{F,\mathfrak l}$,
consists of the elements with nonnegative valuation. We use this
notation for the completed valuation ring throughout. A
\emph{uniformizer} is an element of valuation one.

The equation $\omega^2-\omega-4=0$ reduces modulo two to $X(X-1)$.
The two residue homomorphisms send $\omega$ to $0$ and $1$,
respectively. They have kernels $\qq$ and $\pp$. Hensel lifting
shows that both completions $F_{\qq}$ and $F_{\pp}$ are isomorphic
to $\Q_2$, with distinct embeddings of $F$.

\begin{center}
\renewcommand{\arraystretch}{1.2}
\begin{tabular}{lcc}
\toprule
 & Active prime $\pp$ & Level prime $\qq$\\
\midrule
$\omega$ in the residue field & $1$ & $0$\\
$v_{\mathfrak l}(\rho)$ & $1$ & $0$\\
$v_{\mathfrak l}(\eta)$ & $0$ & $1$\\
Role in the graph & three neighbors & $2^n$ vertices\\
\bottomrule
\end{tabular}
\end{center}
The same quaternion may have determinant of valuation one at $\pp$
and be an integral invertible matrix at $\qq$. This difference in
local behavior is what permits the two roles in the construction.
\section{The quaternion algebra and three involutions}
\label{sec:quaternion}

\subsection{Quaternions and projective involutions}

Our quaternion algebra is the four-dimensional $F$-algebra
\[
 D=F\langle i,j\rangle/(i^2=-1,\ j^2=-3,\ ji=-ij),
 \qquad k=ij.
\]

Its basis is $1,i,j,k$. Quaternion conjugation changes the signs of
$i,j,k$; write $\bar x$ for the conjugate of $x$. The
\emph{reduced trace} is $x+\bar x$ and the \emph{reduced norm} is
$x\bar x$, both in $F$. For $x=x_0+x_1i+x_2j+x_3k$,
\[
 \trd(x)=2x_0,\qquad
 \nrd(x)=x_0^2+x_1^2+3x_2^2+3x_3^2.
\]

The reduced norm is multiplicative, and every element satisfies
\begin{equation}\label{eq:quadraticquaternion}
                x^2-\trd(x)x+\nrd(x)=0.
\end{equation}
The projective group is $D^\times/F^\times$: two nonzero quaternions
define the same projective element when they differ by a nonzero
scalar in $F$. Thus a noncentral trace-zero quaternion with nonzero
norm defines an involution in this group, even though its square in
$D$ need not equal $1$.

At a completion $F_{\mathfrak l}$ we write
$D_{\mathfrak l}=D\otimes_F F_{\mathfrak l}$.
The algebra \emph{splits} at that place if it is isomorphic to
$M_2(F_{\mathfrak l})$. Under such an isomorphism, reduced trace and
norm become matrix trace and determinant. At a place where the
quaternion algebra is a division algebra, it is called
\emph{ramified}. At either real embedding the displayed norm is
positive definite, so $D$ is a division algebra there. This is the
meaning of \emph{totally definite} in this note.

It splits at every finite place. The explicit order below certifies
this through its unit discriminant: the discriminant of a maximal
quaternion order is divisible by every finite ramified prime
\cite[Section~2]{KV}.

\subsection{Orders and an explicit maximal order}

An $\OO_F$-\emph{order} in $D$ is a subring containing $1$ that is
also a full rank-four $\OO_F$-lattice. Here a full lattice means a
finitely generated $\OO_F$-submodule spanning the ambient $F$-vector
space. A \emph{maximal order} is one
not properly contained in any other order. For an order $\mathcal R$,
write $\mathcal R_{\mathfrak l}=
\mathcal R\otimes_{\OO_F}\OO_{F,\mathfrak l}$.
The group $\mathcal R^\times$ consists of elements whose inverses
also belong to the order. Specifying the integral lattice matters:
it determines the congruences that will define the finite graphs.

Define
\begin{equation}\label{eq:maxorder}
 u=\frac{1+j}{2},\qquad
 v=\frac{i-r}{2}+\frac{k+rj}{6},\qquad
 \OO=\OO_F\langle1,i,u,v\rangle.
\end{equation}
Here angle brackets denote the $\OO_F$-span. The multiplication table
\[
\renewcommand{\arraystretch}{1.45}
\begin{array}{c|ccc}
 &i&u&v\\\hline
 i&-1&2r-i-ru+3v&11-ri-6u+rv\\
 u&-2r+2i+ru-3v&u-1&-r+i-v\\
 v&-12+6u-rv&r-i-ru+2v&-6-rv
\end{array}
\]
proves that $\OO$ is an order. For $e=(1,i,u,v)$, direct calculation
gives
\begin{equation}\label{eq:discriminant}
              \det\bigl(\trd(e_s e_t)\bigr)_{s,t=1}^4=-1.
\end{equation}

Here the trace pairing sends $(x,y)$ to $\trd(xy)$. Its determinant
measures the local index of the integral lattice. To see the
maximality test directly, an inclusion of orders multiplies the trace
determinant by the square of the change-of-basis determinant. Trace
pairings on orders are integral. A unit determinant leaves no room
for a nontrivial local index, so the order is maximal. Equivalently,
its reduced discriminant ideal, whose square is the trace
discriminant ideal, is $(1)$. In particular, $D$ is unramified at all
finite primes. The multiplication table and determinant are finite
algebraic checks of the exact integral order used in the proof.

\subsection{The generators}

The following pure quaternions will act as the three automaton states:
\begin{equation}\label{eq:ABC}
 A=\frac{j+\alpha k}{3},\qquad
 B=\frac j3-\frac{\alpha i}{2}-\frac{\alpha k}{6},\qquad
 C=\frac j3+\frac{\alpha i}{2}-\frac{\alpha k}{6}.
\end{equation}
Their order coordinates are
\begin{align}
 A&=(11-2r)-\alpha i+(r-5)u+(r-3)v,\label{eq:coordinates}\\
 B&=(r-6)+\frac{7-r}{2}u-\alpha v,\notag\\
 C&=B+\alpha i.\notag
\end{align}
In particular $A,B,C\in\OO$. The relation
$1+\alpha^2=3\rho$ gives
\begin{equation}\label{eq:generatornorms}
 \trd(A)=\trd(B)=\trd(C)=0,\qquad
 \nrd(A)=\nrd(B)=\nrd(C)=\rho,
 \qquad A^2=B^2=C^2=-\rho.
\end{equation}
Thus the projective classes of $A,B,C$ are involutions. Also
\begin{equation}\label{eq:pairtraces}
 A+B+C=j,\qquad
 \trd(AB)=\trd(AC)=\trd(BC)=\rho-1=-\alpha.
\end{equation}

The residue of $\omega$ at $\qq$ is zero. To check a crucial
congruence in the entire order, rewrite \eqref{eq:coordinates} as
\begin{align*}
 A&=(13-4\omega)+(2-\omega)i+(-6+2\omega)u+(-4+2\omega)v,\\
 B&=(-7+2\omega)+(4-\omega)u+(2-\omega)v,\qquad
 C=B+(-2+\omega)i.
\end{align*}
It follows that
\begin{equation}\label{eq:scalarmod2}
              A\equiv B\equiv C\equiv1\pmod{\qq\OO}.
\end{equation}

\subsection{The local coordinate at the level prime}
\label{subsec:splitting}

Identify $F_{\qq}$ with $\Q_2$. This embedding selects the even root
$\alpha\in\Z_2$ of $X^2+3X-2$, so
$\alpha\equiv10\pmod{16}$. Let $z\in\Z_2$ be the even root of
$z^2+z+2=0$; then $z\equiv2\pmod8$. Hensel's lemma gives both roots
uniquely with these parities. A splitting of $D_{\qq}$ is
\begin{equation}\label{eq:split}
 i\longmapsto\begin{pmatrix}1&1+z\\z&-1\end{pmatrix},
 \qquad
 j\longmapsto\begin{pmatrix}-1&-2\\2&1\end{pmatrix}.
\end{equation}
Multiplication verifies the quaternion relations. In this same
coordinate,
\[
 u\longmapsto\begin{pmatrix}0&-1\\1&1\end{pmatrix},\qquad
 v\longmapsto\frac13
 \begin{pmatrix}2+z-2r&1+2z-r\\z-1+r&-2-z-r\end{pmatrix}.
\]
Every entry is integral. Equation~\eqref{eq:discriminant} shows that
this inclusion of orders is an equality:
\begin{equation}\label{eq:localmaximal}
                         \OO_{\qq}=M_2(\Z_2).
\end{equation}
The arithmetic order and the subsequent coding thus use precisely
the same local lattice.

By \eqref{eq:scalarmod2}, $A,B,C$ lie in the scalar-modulo-$2$
subgroup of $\GL_2(\Z_2)$. Each preserves $x\in2\Z_2$. In the
coordinate $x=2t$, a matrix $g=\left(\begin{smallmatrix}a&b\\c&d
\end{smallmatrix}\right)$ in this subgroup acts as
\begin{equation}\label{eq:mobius}
             \widetilde g(t)=\frac{g(2t)}2
                         =\frac{at+b/2}{2ct+d}.
\end{equation}
Here $a,d$ are odd and $b,c$ are even, so the denominator is odd.
The three induced permutations on $\Z/2^n\Z$ are therefore completely
specified by \eqref{eq:ABC}, \eqref{eq:split}, and \eqref{eq:mobius}.

They commute with reduction modulo powers of two: denominators are
units, and subtraction of two fractional-linear expressions preserves
congruences. Their adjacency sum will be denoted by $C_n$.
This already defines an explicit sequence of cubic arithmetic graphs
on $\Z/2^n\Z$, before any ideal-class or automorphic terminology is
introduced.

\section{The cubic lattice tree and its binary branch}
\label{sec:trees}

\subsection{The local tree}

Let $L_0=\Z_2^2$. A \emph{lattice} in $\Q_2^2$ is a free rank-two
$\Z_2$-submodule spanning $\Q_2^2$. Two lattices are homothetic if
one is a nonzero scalar multiple of the other. The vertices of the
\emph{Bruhat--Tits tree} are their homothety classes $[L]$. Two
classes are adjacent if representatives satisfy
\[
                  2L\subset L'\subset L,\qquad [L:L']=2.
\]
For fixed $L$, the intermediate lattices correspond to the
one-dimensional subspaces of $L/2L\simeq\F_2^2$. There are exactly
three such lines. Hence every vertex has three neighbors.

The elementary-divisor theorem over $\Z_2$ gives a convenient
description of the graph. Every homothety class has the form
$[\Z_2v+2^m\Z_2w]$, where $(v,w)$ is a basis of $L_0$ and
$m\geq0$. Call $m$ its elementary-divisor height relative to
$[L_0]$. Directly calculating the three neighbors shows that a
vertex of height $m>0$ has one neighbor of height $m-1$ and two
of height $m+1$. Thus the graph is connected and has no cycles,
and height is exactly distance from $[L_0]$. The lattice graph
is a tree, and the region beyond a chosen edge is a rooted binary
tree.

The group $\PGL_2(\Q_2)=\GL_2(\Q_2)/\Q_2^\times$ acts by
$[L]\mapsto[gL]$. The stabilizer of $[L_0]$ is
$\PGL_2(\Z_2)$. An integral matrix with determinant of valuation one
has elementary divisors $1,2$, up to units, and therefore moves
$[L_0]$ to a neighbor. At the active prime $\pp$, this applies to
$A,B,C$ because their reduced norm is $\rho$.

\subsection{The branch used for the finite levels}

The ends of this tree are identified with
$\mathbb P^1(\Q_2)$, the one-dimensional subspaces of $\Q_2^2$.
In the affine chart whose column representative is $(x,1)$, the ends
in the branch reducing to $(0,1)$ modulo two are $x\in2\Z_2$.
Writing $x=2t$, the next $n$ binary choices give $t\bmod2^n$.
This is the geometric meaning of the local coordinate
\eqref{eq:mobius}.

By \eqref{eq:scalarmod2}, the three quaternions act as the identity
on the three initial directions at $\qq$. They preserve the selected
branch and act on every finite level within it. The two local roles
can therefore be stated without arithmetic shorthand: the same
three group elements move to neighbors in one cubic tree and
permute a finite binary level in the other. Quaternionic groups
acting on products of trees, their congruence quotients, and the
Hecke interpretation of adjacency are treated in
\cite[Sections~2, 3, and~6]{RSV}.

\section{An explicit conjugacy at every level}
\label{sec:coding}

We now identify the three arithmetic permutations with the three
automaton states. The idea is to imitate prefixing a word by $0$ or
$1$ with two dyadic maps $h_0,h_1$. For example, the automaton rule
$a(0w)=0b(w)$ becomes the identity
$\widetilde A h_0=h_0\widetilde B$. Six such identities will give
the entire recursion. This is the explicit step that connects the
arithmetic construction to this particular automaton.

\subsection{Six identities and two contractions}

Composition of fractional-linear maps is matrix multiplication, and
scalar multiples define the same map. We can therefore search for the
prefix maps directly in the four-dimensional quaternion algebra.
Define
\begin{equation}\label{eq:H}
 H_0=-3\alpha+6i+\alpha j+2k,
 \qquad H_1=-6-3\alpha i+2j-\alpha k.
\end{equation}
Multiplication in $D$, using $\alpha^2=2-3\alpha$, gives
\begin{equation}\label{eq:six}
\begin{aligned}
 AH_0&=-H_0B,& BH_0&=-H_0C,& CH_0&=H_1A,\\
 AH_1&= H_1C,& BH_1&= H_1B,& CH_1&=H_0A.
\end{aligned}
\end{equation}
The two reduced norms are
\begin{equation}\label{eq:Hnorm}
                 \nrd(H_0)=\nrd(H_1)=36(\rho+1).
\end{equation}
These are identities in a four-dimensional algebra over
$\Q[\alpha]/(\alpha^2+3\alpha-2)$, so they require only rational
arithmetic to verify.

In the local splitting \eqref{eq:split}, write $U_s=H_s/4$.
For completeness their matrices are
\begin{align}
 U_0&=\begin{pmatrix}
 2+z-\alpha&1+2z-\alpha/2\\
 z+\alpha/2-1&-2-z-\alpha/2
 \end{pmatrix},\label{eq:U0}\\
 U_1&=\begin{pmatrix}
 -2-\alpha-\alpha z/2&-1-\alpha/2-\alpha z\\
 1+\alpha/2-\alpha z/2&-1+\alpha+\alpha z/2
 \end{pmatrix}.
 \label{eq:U1}
\end{align}
They are integral because $\alpha$ is even. The specified residues
of $\alpha,z$ give
\begin{equation}\label{eq:Umod8}
 U_0\equiv\begin{pmatrix}2&0\\6&7\end{pmatrix}\pmod8,
 \qquad U_1\equiv\begin{pmatrix}2&6\\4&3\end{pmatrix}\pmod8.
\end{equation}
Also $\rho+1=2-\alpha\equiv8\pmod{16}$, so
\[
 v_2(\det U_s)=v_2\bigl(36(\rho+1)/16\bigr)=1.
\]

Write $U_s=\left(\begin{smallmatrix}a_s&b_s\\c_s&d_s
\end{smallmatrix}\right)$ and define
\begin{equation}\label{eq:prefix}
                 h_s(t)=\frac{a_st+b_s/2}{2c_st+d_s}
                 \qquad(t\in\Z_2,\ s\in\{0,1\}).
\end{equation}

\begin{lemma}\label{lem:contractions}
The maps $h_0,h_1$ are bijections onto $2\Z_2$ and $1+2\Z_2$,
respectively. For $t\ne u$,
\begin{equation}\label{eq:metric}
                 v_2\bigl(h_s(t)-h_s(u)\bigr)=1+v_2(t-u).
\end{equation}
\end{lemma}

\begin{proof}
By \eqref{eq:Umod8}, $a_s,b_s,c_s$ are even, $d_s$ is odd,
$a_s\equiv2\pmod4$, and $b_s/2\equiv s\pmod2$. The denominator in
\eqref{eq:prefix} is odd, and $h_s(t)\equiv s\pmod2$.
Subtracting the two fractions gives
\[
 h_s(t)-h_s(u)=
 \frac{\det(U_s)(t-u)}{(2c_st+d_s)(2c_su+d_s)},
\]
which proves \eqref{eq:metric} and injectivity. For
$y\equiv s\pmod2$, the inverse is explicitly
\begin{equation}\label{eq:inverseprefix}
                   h_s^{-1}(y)=\frac{d_s y-b_s/2}{a_s-2c_s y}.
\end{equation}
The numerator is even, and the denominator has valuation one because
$a_s\equiv2\pmod4$ and $c_s$ is even. This lies in $\Z_2$ and
solves $h_s(t)=y$, proving surjectivity.
\end{proof}

The valuation identity says exactly that each $h_s$ multiplies dyadic
distances by $1/2$. After $n$ such maps, changing the input by an
integral amount changes the output by a multiple of $2^n$. This is
the precision statement needed for a proof at every level.

\subsection{The vertex bijections}

For $w=w_1\cdots w_n\in V_n$, set
\begin{equation}\label{eq:kappa}
       \boxed{\displaystyle
       \kappa_n(w)=h_{w_1}\circ h_{w_2}\circ\cdots\circ h_{w_n}(0)
                            \pmod{2^n}.}
\end{equation}
Thus the first letter is the outermost contraction. The final argument
$0$ can be replaced by any element of $\Z_2$ without changing the
residue modulo $2^n$, by \eqref{eq:metric}.

\begin{proposition}\label{prop:coding}
For every $n\geq1$, the map $\kappa_n:V_n\to\Z/2^n\Z$ is a
bijection and satisfies
\begin{equation}\label{eq:intertwinegenerators}
 \kappa_n(a(w))=\widetilde A(\kappa_n(w)),\quad
 \kappa_n(b(w))=\widetilde B(\kappa_n(w)),\quad
 \kappa_n(c(w))=\widetilde C(\kappa_n(w)).
\end{equation}
Moreover,
\begin{equation}\label{eq:prefixcompatibility}
                 \kappa_{n+1}(we)\equiv\kappa_n(w)\pmod{2^n}
                 \qquad(e\in\{0,1\}).
\end{equation}
\end{proposition}

\begin{proof}
The disjoint images in Lemma~\ref{lem:contractions}, followed by
\eqref{eq:metric} within each image, prove by induction that the
$2^n$ words give distinct residues modulo $2^n$. Hence $\kappa_n$
is a bijection. For the compatibility statement, the terminal
argument changes from $0$ to $h_e(0)$ after the same first $n$
contractions; their combined contraction factor is $2^{-n}$.

Pass from \eqref{eq:six} to projective M\"obius transformations.
Scalar signs disappear, and conjugation by the coordinate $x=2t$
gives
\begin{equation}\label{eq:maprecursion}
\begin{array}{lll}
 \widetilde A h_0=h_0\widetilde B,
 &\widetilde B h_0=h_0\widetilde C,
 &\widetilde C h_0=h_1\widetilde A,\\
 \widetilde A h_1=h_1\widetilde C,
 &\widetilde B h_1=h_1\widetilde B,
 &\widetilde C h_1=h_0\widetilde A.
\end{array}
\end{equation}
These are precisely the six transitions in \eqref{eq:recursion}.
For example, pushing $\widetilde A$ through
$h_{w_1}\cdots h_{w_n}$ replaces the word by $a(w)$ and leaves a
terminal state $\widetilde g$, for some $g\in\{A,B,C\}$:
\[
 \widetilde A h_{w_1}\cdots h_{w_n}(0)
          =h_{(a(w))_1}\cdots h_{(a(w))_n}\bigl(\widetilde g(0)\bigr).
\]
The terminal value belongs to $\Z_2$, and the $n$ contractions make
its replacement by $0$ invisible modulo $2^n$. This proves the first
identity in \eqref{eq:intertwinegenerators}; the other two follow in
the same way.
\end{proof}

The inverse coding is equally explicit. Given $t\in\Z_2$, record
its parity $s_1$ and replace it by $h_{s_1}^{-1}(t)$. Record the
new parity $s_2$ and repeat. Lemma~\ref{lem:contractions} guarantees
that each step is defined. This recovers a unique infinite word;
the forward coding recovers $t$ as its compatible sequence of residues.
At finite level, the first $n$ steps recover the unique word
$\kappa_n^{-1}(t\bmod2^n)$.

\subsection{The matrices and the maps between levels}

Let $T_n:\C^{V_n}\to\C^{\Z/2^n\Z}$ be the permutation matrix
\begin{equation}\label{eq:Tentries}
       (T_n)_{t,w}=\begin{cases}
                    1,&t=\kappa_n(w),\\0,&t\ne\kappa_n(w).
                   \end{cases}
\end{equation}

A permutation matrix changes the names of the vertices. In particular,
it preserves the inner product and the eigenvalues of adjacency.
Proposition~\ref{prop:coding} gives
\begin{equation}\label{eq:Tconjugacy}
           C_nT_n=T_nM_n,\qquad T_n^*T_n=T_nT_n^*=I,
           \qquad T_n\one=\one.
\end{equation}
Define pullback maps by
\[
 J_{\rm B,n}e_w=e_{w0}+e_{w1},\qquad
 J_{\rm ar,n}e_t=e_t+e_{t+2^n}.
\]
Equation~\eqref{eq:prefixcompatibility} implies
\begin{equation}\label{eq:towerintertwiner}
                       T_{n+1}J_{\rm B,n}=J_{\rm ar,n}T_n.
\end{equation}
For the alternative Bellaterra convention used in \cite{VV}, the
vertex map is $\kappa_n\circ\delta_n$.

\section{The ideal-class graph and its Hecke operator}
\label{sec:brandt}

The coding has already proved that the automaton graphs and the
explicit residue graphs are isomorphic. To invoke an arithmetic
spectral theorem, we must identify the residue graph with the whole
arithmetic class graph, including every vertex and every neighbor.
That is the purpose of this section.

\subsection{Eichler orders and congruence stabilizers}

An \emph{Eichler order} is the intersection of two maximal orders.
At a split local place, it is conjugate to
\[
 \begin{pmatrix}\OO_{F,\mathfrak l}&\OO_{F,\mathfrak l}\\
 \pi^m\OO_{F,\mathfrak l}&\OO_{F,\mathfrak l}\end{pmatrix}.
\]
Here $\pi$ is a uniformizer and $m\geq0$. The exponent $m$ is its
local level exponent. The product of
$\mathfrak l^m$ over the finite primes is its \emph{level ideal}.
Geometrically, $m$ is the distance between the two vertices whose
endomorphism rings are the maximal orders being intersected. Its
projective unit group fixes that segment pointwise.

For $n\geq0$, specify a global order $\EE_n\subset\OO$ by
\begin{equation}\label{eq:eichler}
 (\EE_n)_{\mathfrak l}=\OO_{\mathfrak l}\quad(\mathfrak l\ne\qq),
 \qquad
 (\EE_n)_{\qq}=\begin{pmatrix}\Z_2&2^{n+1}\Z_2\\2\Z_2&\Z_2\end{pmatrix}.
\end{equation}
These local conditions define the order by intersection in $D$.
Conjugation by $\operatorname{diag}(2^{n+1},1)$ gives
\[
 \begin{pmatrix}2^{-n-1}&0\\0&1\end{pmatrix}
 \begin{pmatrix}a&2^{n+1}b\\2c&d\end{pmatrix}
 \begin{pmatrix}2^{n+1}&0\\0&1\end{pmatrix}
 =\begin{pmatrix}a&b\\2^{n+2}c&d\end{pmatrix}.
\]

Thus $\EE_n$ is an Eichler order of level $\qq^{n+2}$.
The shift by two has a geometric explanation. Put
\[
 L_-=2^{-1}\Z_2e_1+\Z_2e_2,\qquad
 L_+=2^{n+1}\Z_2e_1+\Z_2e_2.
\]
Directly computing their endomorphism rings gives
\[
 \End(L_-)\cap\End(L_+)
    =\begin{pmatrix}\Z_2&2^{n+1}\Z_2\\2\Z_2&\Z_2\end{pmatrix}.
\]
The endpoints have distance $n+2$: one is one edge on the other side
of the base vertex, while the second is $n+1$ edges into the chosen
branch. The local level consequently records $n+2$ edges, although
our finite set has only $n$ binary choices.

At $n=0$, its projective local unit group is
\begin{equation}\label{eq:K2}
 K=\ker\bigl(\PGL_2(\Z_2)\longrightarrow\PGL_2(\F_2)\bigr).
\end{equation}
Indeed, a unit in the matrix order has odd diagonal entries and even
off-diagonal entries, and over $\F_2$ this means reduction to the
identity. The group $K$ acts transitively on $t\bmod2^n$ in
\eqref{eq:mobius}: use the translations
$\left(\begin{smallmatrix}1&2t\\0&1\end{smallmatrix}\right)$.
Its stabilizer of $0\bmod2^n$ is the projective image $K_n$ of
$(\EE_n)_{\qq}^{\times}$. Consequently
\begin{equation}\label{eq:localcosets}
                       K/K_n\simeq\Z/2^n\Z,
                       \qquad gK_n\longmapsto\widetilde g(0).
\end{equation}

\subsection{Ideal classes, class numbers, and the mass}

A \emph{fractional ideal} of $\OO_F$ is an $\OO_F$-lattice in $F$.
The ordinary ideal class group identifies two such ideals when they
differ by multiplication by an element of $F^\times$. The
\emph{narrow} class group permits multiplication only by totally
positive elements. Their cardinalities are the ordinary and narrow
class numbers. Narrow class number one will ensure that the only
exceptional arithmetic eigenfunction is the constant function.

For an order $\mathcal R\subset D$, a \emph{right $\mathcal R$-ideal}
means a full lattice $I\subset D$ with $I\mathcal R\subset I$;
fractional lattices are allowed. It is \emph{locally principal} if
$I_{\mathfrak l}=x_{\mathfrak l}\mathcal R_{\mathfrak l}$ for some
$x_{\mathfrak l}\in D_{\mathfrak l}^\times$ at every finite prime.
Two such ideals represent the same class if $I=dJ$ for a
$d\in D^\times$. Multiplication is on the left because these are
right ideals. Their finite class set is $\Cl(\mathcal R)$ and its
cardinality is $h(\mathcal R)$.

The \emph{left order} of $I$ is
$\mathcal R_I=\{x\in D:xI\subseteq I\}$. Its projective unit
group is $\mathcal R_I^\times/\OO_F^\times$. For locally principal
ideals of a definite order these groups are finite. The
\emph{mass} is the weighted count
\[
 \operatorname{Mass}(\mathcal R)
       =\sum_{[I]\in\Cl(\mathcal R)}
                       \frac{1}{|\mathcal R_I^\times/\OO_F^\times|}.
\]
The mass formula is a standard arithmetic input
\cite[Section~5]{KV}. For a totally definite quaternion algebra over
a totally real field of degree $d$, with finite discriminant $(1)$,
and an Eichler order of level $\mathfrak N$, its specialization is
\begin{equation}\label{eq:generalmass}
 \operatorname{Mass}(\mathcal R)=
  2^{1-d}h_F|\zeta_F(-1)|\,\Nm\mathfrak N
     \prod_{\mathfrak l\mid\mathfrak N}
                       (1+\Nm\mathfrak l^{-1}).
\end{equation}
Here $\zeta_F$ is the Dedekind zeta function, initially
$\sum_{\mathfrak a\ne0}(\Nm\mathfrak a)^{-s}$ over nonzero
integral ideals for $\Re(s)>1$, continued to $s=-1$. In our case its value is elementary to compute.
We will also prove that every weight in the mass is one, so the mass
is the actual number of vertices.

\begin{lemma}\label{lem:unitsmass}
The field $F$ has narrow class number one and $\zeta_F(-1)=1/3$.
Every Eichler order of level $\qq^m$ with $m\geq2$ has only scalar
units. Its ideal class number is $2^{m-2}$; in particular,
\begin{equation}\label{eq:baseclassone}
                  h(\EE_0)=1,\qquad \EE_0^\times=\OO_F^\times.
\end{equation}
\end{lemma}

\begin{proof}
The quadratic Minkowski bound is $\sqrt{17}/2<2.1$. Every ideal class
therefore has an integral representative of norm at most $2$, and
the two primes of norm $2$ are the principal ideals $(\eta)$ and
$(\rho)$. Hence the ordinary class number is one.

Put $\epsilon=4+r$, of norm $-1$. We give an elementary unit
calculation. If a unit has first real value $1<u<\epsilon$, its
other real value is $u'=\pm u^{-1}$. Writing $u=(a+br)/2$ gives
$0<br=u-u'<2r$, because $\epsilon+\epsilon^{-1}=2r$ and
$\epsilon-\epsilon^{-1}=8<2r$. Thus $b=1$, but the unit norm
equation would give $a^2=13$ or $21$. There is no such unit.
Dividing any unit positive at the first embedding by a suitable
power of $\epsilon$ places that embedding in $[1,\epsilon)$.
Consequently
\[
               \OO_F^\times=\{\pm\epsilon^m:m\in\Z\}.
\]
The second embedding of $\epsilon$ is negative, so the totally
positive units are precisely $\epsilon^{2m}$, hence squares.
The sign patterns of $-1$ and $\epsilon$ also show that every
principal ideal has a totally positive generator. The narrow class
number is therefore one.

For completeness, let $\chi(a)=(\frac a{17})$ be the quadratic
character: it is $1$ on nonzero squares modulo $17$, $-1$ on
nonsquares, and $0$ on multiples of $17$. Its Dirichlet function is
$L(s,\chi)=\sum_{a\geq1}\chi(a)a^{-s}$, initially for $\Re(s)>1$.
The quadratic-field factorization is
$\zeta_F(s)=\zeta_\Q(s)L(s,\chi)$. The elementary special-value
formula and finite sum give
\[
 L(-1,\chi)=-\frac1{34}\sum_{a=1}^{16}\chi(a)a^2=-4,
 \qquad \zeta_F(-1)=(-1/12)(-4)=1/3.
\]

Let $\mathcal R$ be any Eichler order of the indicated level and let
$x\in\mathcal R^\times$. Its reduced norm is a totally positive
unit. Multiplying by a scalar unit, normalize $\nrd(x)=1$. If $x$
is noncentral, definiteness implies $|\sigma(\trd(x))|<2$ at both
real embeddings. Write the integral trace as
$t=(a+br)/2$, where $a,b\in\Z$ have the same parity. The difference
between its two embeddings is $br$, so $|b|r<4$ forces $b=0$.
Thus $t\in\{-1,0,1\}$.

Locally at $\qq$, conjugate $\mathcal R$ to the standard Eichler
order, which is upper triangular modulo $4$. A determinant-one
unit has odd diagonal entries with product $1$ modulo $4$.
Their sum is $2$ modulo $4$, so its trace cannot be $-1$, $0$, or
$1$. This contradiction proves that every unit is central. The same
argument applies to every left order of a locally principal ideal.

The definite Eichler mass formula \cite[Section~5]{KV} now counts
ideal classes without weights, because each projective unit group is
trivial. The finite discriminant of $D$ is $(1)$, so the mass is
\begin{equation}\label{eq:mass}
 h(\mathcal R)=2^{1-[F:\Q]}h_F|\zeta_F(-1)|
         \Nm(\qq^m)\left(1+\frac1{\Nm\qq}\right)
       =\frac12\cdot\frac13\cdot2^m\cdot\frac32
       =2^{m-2}.
\end{equation}
This proves all assertions.
\end{proof}

\subsection{Finite ideles as a notation for local lattices}

A \emph{finite idele} of $D$ is a tuple
$x=(x_{\mathfrak l})_{\mathfrak l}$, with
$x_{\mathfrak l}\in D_{\mathfrak l}^\times$, that belongs to the
local order unit group at all but finitely many primes.
These tuples form a group $\wh D^\times$. The global group
$D^\times$ embeds diagonally. Put
\[
 \wh{\mathcal R}=\prod_{\mathfrak l}\mathcal R_{\mathfrak l},
 \qquad\wh{\mathcal R}^{\times}=
                     \prod_{\mathfrak l}\mathcal R_{\mathfrak l}^{\times}.
\]
The tuple $x$ determines the right ideal
$I_x=D\cap x\wh{\mathcal R}$. Right multiplication by local order
units leaves this ideal unchanged; left multiplication by a global
element changes it within its ideal class. The local-to-global
description of lattices therefore gives the double-coset formula
\begin{equation}\label{eq:idelicclasses}
       \Cl(\mathcal R)=D^\times\backslash\wh D^\times/
                                      \wh{\mathcal R}^{\times}.
\end{equation}
This is bookkeeping for local generators of ideals. A tuple need not
be $1$ at almost every prime; it need only be a unit there. The
particular representatives used below can be chosen to be $1$ away
from the level prime.

\subsection{All classes and all three Hecke neighbors}

Put $\mathcal U_n=\wh\EE_n^\times$. A \emph{Brandt module} is the vector
space of functions on $\Cl(\EE_n)$. The \emph{Hecke operator} at a
prime away from the level sends a function to the sum of its values
on the local neighboring lattices at that prime, leaving the local
data at every other prime unchanged. Its matrix in the ideal-class
basis is called a \emph{Brandt matrix} \cite[Section~4]{DV}.
The directed graph with vertices $\Cl(\EE_n)$ and one edge for each
neighbor, counted with its multiplicity, is called a \emph{Brandt graph}.
For a split prime with residue field of size $q$, the operator has
$q+1$ neighbor terms. We always use this unnormalized neighbor sum.

In matrix notation, the local operator is the double coset
$\GL_2(\OO_{F,\pp})\operatorname{diag}(\rho,1)
 \GL_2(\OO_{F,\pp})$.
Its right cosets parametrize the $\Nm\pp+1=3$ neighbors described
in Section~\ref{sec:trees}. For each vertex, right multiplication
of its idele representative by these cosets gives its three
arithmetic neighbors. Distinct neighbors may represent the same
ideal class, producing multiplicities or fixed-state entries.
The next proposition identifies this entire matrix.

\begin{proposition}\label{prop:classes}
For every $n\geq0$, there are bijections
\begin{equation}\label{eq:classbijection}
 \Cl(\EE_n)\simeq \mathcal U_0/\mathcal U_n\simeq K/K_n\simeq\Z/2^n\Z.
\end{equation}
Under these bijections, the unnormalized Hecke operator at $\pp$ is
exactly $C_n$, the sum of the permutations induced by $A,B,C$.
\end{proposition}

\begin{proof}

The following elementary coset principle is useful. If $G=H\mathcal U_0$
for subgroups $H,\mathcal U_0$, and $U\subseteq \mathcal U_0$, then
$H\backslash G/U\cong(H\cap \mathcal U_0)\backslash \mathcal U_0/U$.
Indeed, every orbit under left multiplication by $H$ has a
representative in $\mathcal U_0$. If $u,u'\in\mathcal U_0$ represent
the same double coset, then $u'=huv$ with $h\in H$, $v\in U$.
Hence $h=u'v^{-1}u^{-1}\in H\cap\mathcal U_0$, which proves
both injectivity and the asserted equivalence relation.

Since $h(\EE_0)=1$, formula~\eqref{eq:idelicclasses} says precisely
that $\wh D^\times=D^\times \mathcal U_0$. Applying this principle gives
\[
 \Cl(\EE_n)\simeq(D^\times\cap \mathcal U_0)\backslash \mathcal U_0/\mathcal U_n
                       =\EE_0^\times\backslash \mathcal U_0/\mathcal U_n.
\]
The group $\EE_0^\times=\OO_F^\times$ is central and contained in
$\mathcal U_n$, so its left action on $\mathcal U_0/\mathcal U_n$ is trivial. The only local
factor in which $\mathcal U_n$ differs from $\mathcal U_0$ is $\qq$. Dividing both local
groups by their common center and applying \eqref{eq:localcosets}
proves the bijections. Explicit representatives are the finite ideles
$u_t$, for $t\bmod2^n$, whose component at $\qq$ is
\[
                 (u_t)_{\qq}=\begin{pmatrix}1&2t\\0&1\end{pmatrix}
\]
and whose other components are $1$. The corresponding right ideal is
$I_t=D\cap u_t\wh\EE_n$. In particular, every class has a representative
$u\in \mathcal U_0$ supported only at $\qq$.

We now compute the Hecke action, including its coset direction.
At $\pp$, identify $\OO_{\pp}$ with $M_2(\OO_{F,\pp})$.
The elements $A,B,C$ are integral of determinant valuation one. They
therefore belong to the Hecke double coset of
$\operatorname{diag}(\rho,1)$. They represent three distinct right
cosets of $\OO_{\pp}^{\times}$. Indeed, for distinct $g,h$ among
$A,B,C$, \eqref{eq:pairtraces} gives $\trd(gh)=\rho-1$, a
$\pp$-adic unit, while $\nrd(gh)=\rho^2$. Thus the elementary
divisors of $gh$ have valuations $0,2$. As
$g^{-1}h=-gh/\rho$, it is not an integral invertible matrix.
Hence $g\OO_{\pp}^\times\ne h\OO_{\pp}^\times$. There are precisely
$\Nm\pp+1=3$ such cosets, so these are all of them.

For $g\in\{A,B,C\}$, let $t_g\in\wh D^\times$ equal $g$ at
$\pp$ and $1$ at every other finite place. Write
\[
 t_g=gz_g,\qquad (z_g)_{\pp}=1,\qquad
                    (z_g)_{\mathfrak l}=g^{-1}
                    \quad(\mathfrak l\ne\pp).
\]
The element $z_g$ belongs to $\mathcal U_0$: away from $\pp$ each generator
is an order unit, and at $\qq$ it is scalar modulo $2$.
For a representative $u$ supported only at $\qq$, the finite ideles
$u$ and $t_g$ commute. Therefore
\begin{equation}\label{eq:heckecoset}
                   [u t_g]=[t_g u]=[g z_g u]=[z_g u]
                         \quad\hbox{in }\Cl(\EE_n).
\end{equation}
In the quotient $K/K_n$, this is left multiplication by
$g_{\qq}^{-1}$. But $g^{-1}=-g/\rho$, and $-1/\rho$ is a scalar
unit at $\qq$, so it induces precisely the same permutation as $g$.
Thus the three Hecke neighbors of a residue $t$ are
$\widetilde A(t)$, $\widetilde B(t)$, and $\widetilde C(t)$,
with their multiplicities. This proves $T_{\pp}=C_n$.
\end{proof}

\begin{remark}[Torsion and the tree]\label{rem:torsion}

The notation $\EE_0[1/\rho]$ allows denominators that are powers
of $\rho$, changing integrality only at $\pp$. Let
$\operatorname P$ denote image in $D^\times/F^\times$.
The arithmetic group
$\Gamma=\operatorname P(\EE_0[1/\rho]^\times)$ acts simply
transitively on the vertices of the cubic Bruhat--Tits tree at $\pp$.
To see transitivity, prescribe a locally principal right ideal that
is $\EE_0$ away from $\pp$ and represents a chosen tree vertex at
$\pp$; equation~\eqref{eq:baseclassone} makes it principal.
A base-vertex stabilizer becomes an element of $\EE_0^\times$ after
multiplication by a scalar power of $\rho$, hence is projectively
trivial. The three generators give the three distinct neighbors by
the elementary-divisor argument above. They generate $\Gamma$, and
a reduced word traces a nonbacktracking path. Consequently
\[
                    \Gamma=\langle[A],[B],[C]\rangle
                              \simeq C_2*C_2*C_2.
\]

Here $C_2$ is the cyclic group of order two and $*$ is free product.
Its usual Cayley graph with these three involutions is the cubic tree.
The arithmetic realization of this free product has the additional
structure of the second local action, which makes its finite
Schreier graphs accessible to the Hecke bound. The abstract free
product presentation alone would not impose that bound on arbitrary
finite permutation actions.

These involutions invert edges. Their torsion is compatible with the
construction: fixed states in a finite quotient give the diagonal
contributions in \eqref{eq:adjacency}. All vertex weights are one by
Lemma~\ref{lem:unitsmass}.
\end{remark}

\section{The Ramanujan bound}
\label{sec:ramanujan}

\subsection{The arithmetic theorem and the meaning of its terms}

We state the needed spectral input in the form that matches our
finite graph. This also specifies the point at which the proof uses
results beyond elementary group and local matrix calculations.

The theorem below combines Jacquet--Langlands, in the form explained in
\cite[Theorem~3.9]{DV}, with the Hilbert modular Ramanujan theorem
\cite[Theorem~1]{Blasius}.

\begin{theorem}[Arithmetic spectral input]\label{thm:arithmeticinput}
Let $F$ be a totally real field of narrow class number one, and let
$D/F$ be a quaternion algebra ramified at every real place and at no
finite place. Let $\mathcal R$ be an Eichler order of level
$\mathfrak N$. The orthogonal complement of the constant functions in its scalar
Brandt module transfers, with the Hecke action away from $\mathfrak N$, to
parallel-weight-two Hilbert cusp forms with trivial central character.
For a prime $\mathfrak l\nmid\mathfrak N$, every corresponding
unnormalized Hecke eigenvalue satisfies
\[
                      |\lambda_{\mathfrak l}|
                               \leq2\sqrt{\Nm\mathfrak l}.
\]
This orthogonal complement is taken with respect to the standard
Brandt inner product. If the projective unit groups of the left orders
are all trivial, this is the ordinary inner product on the class set.
\end{theorem}

A \emph{Hilbert modular form} over a totally real field is the
analogue of a modular form with one upper-half-plane variable for
each real embedding. \emph{Parallel weight two} means weight two
in every variable; for $F=\Q(\sqrt{17})$ the weight is $(2,2)$.
\emph{Cuspidal} means that the form vanishes at the cusps, the
boundary points in the modular interpretation. The
Jacquet--Langlands correspondence identifies the non-character
part of a definite quaternionic module with such cusp forms while
preserving the relevant Hecke eigenvalues. This is the analytic
interpretation of the finite-dimensional vector space used here.

The scalar quaternionic module consists of complex-valued functions
on the ideal classes; no extra representation is attached to a
vertex. Its possible one-dimensional exceptional constituents are
\emph{norm characters}, functions obtained by applying a character
of a field idele-class group to the reduced norm of an idele.
The class-number condition will make all such functions constant.
The \emph{central character} records the action of scalar ideles;
it is trivial when every scalar idele acts as the identity.

At a prime away from the level, choose a simultaneous Hecke
eigenvector in the cuspidal subspace with eigenvalue $\lambda_{\mathfrak l}$. Its two
\emph{Satake parameters}, in the weight-two arithmetic normalization,
are the roots $\beta_{\mathfrak l},\gamma_{\mathfrak l}$ of
\[
                 X^2-\lambda_{\mathfrak l}X+\Nm\mathfrak l.
\]
The Ramanujan theorem says that both roots have complex absolute
value $\sqrt{\Nm\mathfrak l}$. Their sum is the eigenvalue, so the
displayed bound follows by the triangle inequality. In normalized
adjacency the same bound is divided by $\Nm\mathfrak l+1$.
For comparison, the constant function has eigenvalue
$\Nm\mathfrak l+1$, and its polynomial has roots $1$ and
$\Nm\mathfrak l$. This explains why identifying and removing
the constant constituent is part of the proof.

\subsection{Checking the hypotheses and proving the graph theorem}

\begin{proof}[Proof of Theorem~\ref{thm:main}]
By Proposition~\ref{prop:classes}, the entire operator $C_n$ is the
unnormalized Hecke operator $T_{\pp}$ on the ideal classes of the
Eichler order $\EE_n$, of level $\qq^{n+2}$. It has three terms at
every vertex. The projective unit groups of all left orders are
trivial by Lemma~\ref{lem:unitsmass}; hence the usual inner product
on this Brandt module is a constant multiple of the ordinary
Euclidean inner product.

We identify the noncuspidal part. The one-dimensional automorphic
constituents of a definite quaternionic module factor through the
reduced norm; see \cite[Definition~3.7 and equation~(3.8)]{DV} or
\cite[Section~7]{KV}. At every finite place, the reduced norm of the
Eichler unit group is the full local unit group: after local
conjugation it contains every matrix $\operatorname{diag}(u,1)$
with $u$ a local unit. Thus such characters factor through the narrow
ideal class group of $F$. This group is trivial by
Lemma~\ref{lem:unitsmass}. The noncuspidal subspace is therefore
exactly $\C\one$, and its Hecke eigenvalue is $3$.

The orthogonal complement $\one^\perp$ is the scalar cuspidal
quaternionic module. The algebra $D$ is ramified at both real places
and at no finite place. The Jacquet--Langlands correspondence
\cite[Theorem~3.9]{DV} identifies this module, with its Hecke action
away from the level, with parallel-weight-$(2,2)$ Hilbert cusp forms
of level $\qq^{n+2}$. The central character is trivial: finite
central ideles act trivially on the ideal-class module because
$h_F=1$ and local central units lie in $\mathcal U_n$.

The active prime $\pp$ is distinct from $\qq$, so the transferred
representations are unramified at $\pp$. For a simultaneous Hecke
eigenvector, let $\beta_\pp,\gamma_\pp$ denote the Satake parameters
in the weight-two arithmetic normalization. They satisfy
\[
 \beta_\pp+\gamma_\pp=\lambda_\pp,\qquad
 \beta_\pp\gamma_\pp=\Nm\pp=2.
\]
The Hilbert modular Ramanujan theorem \cite[Theorem~1]{Blasius}
gives
\[
                  |\beta_\pp|=|\gamma_\pp|=\sqrt2.
\]
Its hypotheses include weights $(2,2)$ over this degree-two totally
real field, and its conclusion applies at all finite places,
including dyadic places. Therefore
\begin{equation}\label{eq:Heckebound}
                |\lambda_\pp|\leq2\sqrt2
                     \quad\hbox{on }\one^\perp.
\end{equation}
Equivalently, one can apply the Ramanujan statement for these prime
ideal graphs in \cite[Theorem~7.6]{KV}; here
$\pp\nmid (17)\operatorname{disc}(D)\qq^{n+2}$.

Finally, \eqref{eq:Tconjugacy} conjugates $M_n$ to $C_n$ by a
permutation matrix preserving constants. Equation~\eqref{eq:Heckebound}
therefore gives \eqref{eq:bound}. Since $3>2\sqrt2$, the eigenspace
of $3$ is exactly the one-dimensional constant subspace. A finite
symmetric nonnegative matrix with all row sums $3$ has this property
if and only if its graph is connected. This proves connectedness
and the simplicity of the eigenvalue $3$, and completes the theorem. 
Compatibility of the intertwiners is \eqref{eq:towerintertwiner}.
\end{proof}

\begin{corollary}
The normalized adjacency $P_n=M_n/3$ satisfies
\[
                  \|P_n|_{\one^\perp}\|\leq\frac{2\sqrt2}{3}<1
                  \qquad(n\geq1).
\]
Thus the Bellaterra level graphs have a uniform absolute spectral gap
of at least $1-2\sqrt2/3$.
\end{corollary}

\begin{corollary}[The explicit arithmetic realization]
With the parameters \eqref{eq:parameters}, the order
\eqref{eq:eichler}, and the permutation matrix \eqref{eq:Tentries},
\begin{equation}\label{eq:mainintertwiner}
          C_nT_n=T_nM_n,\qquad
          C_n=T_{\pp}\text{ on }\C^{\Cl(\EE_n)}.
\end{equation}
The order has level $\qq^{n+2}$ and exactly $2^n$ ideal classes.
The Hecke operator has row sum $\Nm\pp+1=3$. The matrices $T_n$
give one compatible, color-preserving isomorphism of graph towers.
\end{corollary}

The six recursion identities give the word-to-residue map, and the
class-number-one argument identifies the space of functions on all ideal classes.
Together they identify the Bellaterra adjacency with the Hecke
operator to which the arithmetic bound applies.

\section{Acknowledgement}

The paper was written with Chat GPT 6 Pro editing assistance and assistance with finding known results needed for 
Theorem 8.1. 
The content of the paper was carefully checked by the authors
and the authors carry the full responsibility for the content of the paper. The first author thanks Dolciani-Holloran foundation for the support. The second author thanks PSC CUNY for the support.


\begin{thebibliography}{RSV19}
\small
\setlength{\itemsep}{5pt}

\bibitem[Bla06]{Blasius}
Don Blasius,
\emph{Hilbert modular forms and the Ramanujan conjecture},
in C.~Consani and M.~Marcolli (eds.),
\emph{Noncommutative Geometry and Number Theory},
Aspects of Mathematics E37, Vieweg, Wiesbaden, 2006, pp.~35--56.
\href{https://doi.org/10.1007/978-3-8348-0352-8_2}
{doi:10.1007/978-3-8348-0352-8\_2}.

\bibitem[BGV26]{BGV}
Ievgen Bondarenko, Rostislav Grigorchuk, and Alina Vdovina,
\emph{Ramanujan subshifts}, 2026.
\href{https://arxiv.org/abs/2602.22356}{arXiv:2602.22356}.

\bibitem[DV13]{DV}
Lassina Demb\'el\'e and John Voight,
\emph{Explicit methods for Hilbert modular forms},
in \emph{Elliptic Curves, Hilbert Modular Forms and Galois Deformations},
Advanced Courses in Mathematics--CRM Barcelona,
Birkh\"auser, Basel, 2013, pp.~135--198.
\href{https://doi.org/10.1007/978-3-0348-0618-3_4}
{doi:10.1007/978-3-0348-0618-3\_4}.

\bibitem[Gri05]{Gri05}
Rostislav Grigorchuk,
\emph{Solved and Unsolved Problems Around One Group},
in L.~Bartholdi, T.~Ceccherini-Silberstein, T.~Smirnova-Nagnibeda,
and A.~\.{Z}uk (eds.),
\emph{Infinite Groups: Geometric, Combinatorial and Dynamical Aspects},
Progress in Mathematics \textbf{248}, Birkh\"auser, Basel, 2005,
pp.~117--218.
\href{https://doi.org/10.1007/3-7643-7447-0_5}
{doi:10.1007/3-7643-7447-0\_5}.

\bibitem[Gri11]{Gri11}
R.~I.~Grigorchuk,
\emph{Some Topics in the Dynamics of Group Actions on Rooted Trees},
Proceedings of the Steklov Institute of Mathematics \textbf{273}
(2011), 64--175.
\href{https://doi.org/10.1134/S0081543811040067}
{doi:10.1134/S0081543811040067}.

\bibitem[KV10]{KV}
Markus Kirschmer and John Voight,
\emph{Algorithmic enumeration of ideal classes for quaternion orders},
SIAM Journal on Computing \textbf{39} (2010), no.~5, 1714--1747.
\href{https://doi.org/10.1137/080734467}{doi:10.1137/080734467}.
Corrigendum, \textbf{41} (2012), no.~3, 714.
\href{https://doi.org/10.1137/120866063}{doi:10.1137/120866063}.

\bibitem[MP17]{MP}
Anton Malyshev and Igor Pak,
\emph{Lifts, derandomization, and diameters of Schreier graphs of
Mealy automata}, Combinatorica \textbf{37} (2017), 733--765.
\href{https://doi.org/10.1007/s00493-016-3306-0}
{doi:10.1007/s00493-016-3306-0}.

\bibitem[RSV19]{RSV}
Nithi Rungtanapirom, Jakob Stix, and Alina Vdovina,
\emph{Infinite series of quaternionic 1-vertex cube complexes, the
doubling construction, and explicit cubical Ramanujan complexes},
International Journal of Algebra and Computation \textbf{29} (2019),
no.~6, 951--1007.
\href{https://doi.org/10.1142/S0218196719500371}
{doi:10.1142/S0218196719500371}.

\bibitem[VV10]{VV}
Mariya Vorobets and Yaroslav Vorobets,
\emph{On a series of finite automata defining free transformation groups},
Groups, Geometry, and Dynamics \textbf{4} (2010), no.~2, 377--405.
\href{https://doi.org/10.4171/GGD/87}{doi:10.4171/GGD/87}.

\end{thebibliography}
\end{document}